\def\Filedate{2011/01/24} 
\def\Fileversion{1.15}  

\documentclass{compositio}

\usepackage{mathpazo}
\usepackage{euler}

\usepackage[inline]{enumitem}
\usepackage{subcaption}

\usepackage{amssymb,amsmath,latexsym,lipsum}
\usepackage{enumitem}
\usepackage{tabulary}
\usepackage{booktabs}
\usepackage{ upgreek }
\usepackage{charter}
\usepackage[title]{appendix}
\usepackage{longtable}
\usepackage{amsthm}
\usepackage{euscript}
\usepackage{color}
\usepackage[T1]{fontenc}
\usepackage[singlespacing]{setspace}
\usepackage[colorinlistoftodos]{todonotes}
\usepackage{tikz}
\usetikzlibrary {positioning}
\definecolor {processblue}{cmyk}{0.96,0,0,0}
\usepackage{bookmark}
\usepackage{multirow,bigdelim}
\usepackage{stackengine}
\usepackage{mathrsfs}
\usepackage{mathtools}
\usepackage{tikz-cd}
\usepackage{pgfplots}
\usepackage{pgf}
\usetikzlibrary{arrows}

\pgfplotsset{compat=1.16}
\newcommand*{\pkg}[1]{{\mdseries\textsf{#1}}}

\renewcommand{\contentsname}{Contents\\{\footnotesize\normalfont}}

\newtheorem{theorem}{Theorem}[section]
\newtheorem{proposition}[theorem]{Proposition}
\newtheorem{lemma}[theorem]{Lemma}
\newtheorem{corollary}[theorem]{Corollary}
\newtheorem{definition}[theorem]{Definition}
\theoremstyle{remark}
\newtheorem{example}[theorem]{Example}
\theoremstyle{remark}
\newtheorem{remark}[theorem]{Remark}
\newtheorem{question}[theorem]{Question}

\DeclareMathOperator{\spec}{Spec}

\DeclareMathOperator{\sw}{sw}
\DeclareMathOperator{\Frac}{Frac}
\DeclareMathOperator{\cov}{cov}
\DeclareMathOperator{\divisor}{div}
\DeclareMathOperator{\ram}{ram}

\DeclareMathOperator{\Image}{Im}
\DeclareMathOperator{\C}{\mathcal{C}}
\DeclareMathOperator{\ASW}{\mathcal{ASW}}
\DeclareMathOperator{\Proj}{Proj}
\DeclareMathOperator{\gal}{Gal}

\DeclareMathOperator{\MggpG}{\mathcal{M}_{g,g',k}[G]}

\DeclareMathOperator{\Mgg}{\mathcal{M}_{g,g',k}}

\DeclareMathOperator{\cohom}{H}

\begin{document}

\title{The moduli space of cyclic covers in positive characteristic}

\author{Huy Dang}
\email{huydang1130@ncts.ntu.edu.tw}
\address{National Center for Theoretical Sciences, Mathematics Division, No. 1, Sec. 4, Roosevelt Rd., Taipei City 106, Taiwan Room 503, Cosmology Building, National Taiwan University}

\author{Matthias Hippold}
\email{matthias.hippold@mail.huji.ac.il}
\address{Einstein Institute of Mathematics, The Hebrew University of Jerusalem, Edmond J.
Safra Campus, Giv’at Ram, Jerusalem, 91904, Israel }

\thanks{This file documents \pkg{compositio} version \Fileversion\ and
was last revised \Filedate.}

\begin{abstract}
We study the $p$-rank stratification of the moduli space $\ASW_{(d_1,d_2,\ldots,d_n)}$, which represents $\mathbb{Z}/p^n$-covers in characteristic $p>0$ whose $\mathbb{Z}/p^i$-subcovers have conductor $d_i$. In particular, we identify the irreducible components of the moduli space and determine their dimensions. To achieve this, we analyze the ramification data of the represented curves and use it to classify all the irreducible components of the space.

In addition, we provide a comprehensive list of pairs $(p,(d_1,d_2,\ldots,d_n))$ for which $\ASW_{(d_1,d_2,\ldots,d_n)}$ in characteristic $p$ is irreducible. Finally, we investigate the geometry of $\ASW_{(d_1,d_2,\ldots,d_n)}$ by studying the deformations of cyclic covers which vary the $p$-rank and the number of branch points.
\end{abstract}
 
\maketitle

\vspace*{6pt}\tableofcontents

\section{Introduction}

Consider an algebraically closed field $k$ of characteristic $p > 0$. An $n$-Artin-Schreier-Witt $k$-curve, or simply an Artin-Schreier-Witt curve when $n$ is fixed, refers to a smooth, projective, connected $k$-curve $Y$ that serves as a $\mathbb{Z}/p^n$-cover of the projective line. It is worth noting that any Galois cover can be split into its tame and wild parts. The tame part can be understood through the well-known Kummer theory. Hence, studying these Artin-Schreier-Witt curves is sufficient for comprehending all cyclic covers. Moreover, these curves can be effectively constructed and analyzed using explicit methods as they are determined by a (class of) length-$n$ Witt vector over $k$. Additionally, the genus of the covering and the $p$-rank of its Jacobian can be easily derived from the represented Witt vector, as discussed in Section \ref{secASWcurves}.

When $n=1$, $Y$ is called an Artin-Schreier curve and has been extensively studied. Pries and Zhu defined the moduli space $\mathcal{AS}_g$ that parametrizes Artin-Schreier curves of fixed genus $g$ and established many nice properties \cite{MR2985514}. In this paper, we introduce the moduli space $\ASW_{\underline{\iota}}$ ($\underline{\iota} = (\iota_1, \ldots, \iota_n) \in \mathbb{Z}^n$), which represents $n$-Artin-Schreier-Witt curves whose $\mathbb{Z}/p^i$ sub-covering has a genus $g_i$ associated with $\iota_i$. We study its stratification by $p$-rank into strata denoted by $\ASW_{\underline{\iota}, \underline{s}}$, where $\underline{s} = (s_1, \ldots, s_n) \in \mathbb{Z}^n$ determines the $p$-ranks. Points in these strata correspond to curves whose $\mathbb{Z}/p^i$ sub-covering has a $p$-rank of $s_i$. Our main result is the following.

\begin{theorem}
\begin{enumerate}
\item The irreducible components of $\ASW_{\underline{\iota}}$ are in bijection with the set of partitions $M=(\iota_{j,i}) \in \mathbb{M}_{r\times n}(\mathbb{Z}_{\ge 0})$ of $\underline{\iota}$ with no ``essential parts'' (see Section \ref{secessentialdeform}).
\item The irreducible component of $\ASW_{\underline{\iota}, \underline{s}}$ associated with the partition $M=(\iota_{j,i}) \in \mathbb{M}_{r\times n}(\mathbb{Z}_{\ge 0})$ has dimension given by
\begin{equation*}
s_n+\sum_{i=1}^n \sum_{j=1}^r \left( \iota_{j,i} -1 - \left\lfloor \frac{\iota_{j,i}-1}{p} \right\rfloor \right).
\end{equation*}
\end{enumerate}
\end{theorem}

As an application, we determine all cases when $\ASW_{\underline{\iota}}$ is irreducible.

\begin{corollary}
The moduli space $\ASW_{(\iota_1, \ldots, \iota_n)}$ is irreducible if and only if the following conditions hold:
\begin{enumerate}[label*=\arabic*.]
\item $\iota_1 = 2$ or $\iota_1 = 3$.
\item $\iota_{i} = p\iota_{i-1}-p+1$ or $\iota_{i} = p\iota_{i-1}-p+2$ for $2 \leq i \leq n$.
\end{enumerate}
\end{corollary}

Another fascinating aspect of $\ASW_{\underline{\iota}}$ is the existence of coverings corresponding to certain closed points in the moduli space, where the number of branch points varies while remaining in the same flat family. In other words, it is possible to construct a flat deformation of such a cover, in which the number of branch points changes while the genus remains constant. This phenomenon was first studied by M{\'e}zard in her thesis \cite{MEZAR1998} and was later utilized by Pop in the solution to the Oort conjecture \cite{MR3194816}. Recently, the first author has developed a technique employing the concept of Hurwitz tree (introduced by Henrio \cite{2000math.....11098H} and enhanced by Bouw, Brewis, Wewers \cite{MR2254623} \cite{MR2534115}) to determine the existence of a deformation with specified ramification data \cite{2020arXiv200203719D, 2020arXiv201013614D}. In this manuscript, we will utilize the main result from that work to establish relations between the moduli spaces.
\begin{proposition}
    The canonical morphism $\ASW_{(\iota_1, \ldots, \iota_{n-1}, \iota_n, \ldots, \iota_m)} \xrightarrow{} \ASW_{(\iota_1, \ldots, \iota_{n-1}, \iota_n)}$ is closed.
\end{proposition}
As a consequence, we can prove some results regarding the (dis)connectedness of $\ASW_{\underline{\iota}}$.
\begin{corollary}
For $p \geq 5$ and $3 \leq \iota_1 \leq 2p-2$, the moduli space $\ASW_{(\iota_1, \ldots, \iota_n)}$ is disconnected.
\end{corollary}
\subsection{Outline} This manuscript presents a generalization of the work in \cite{MR2985514} and closely follows its structure. In Section \S \ref{secpartionASW}, we provide a comprehensive overview of Artin-Schreier-Witt theory and introduce a combinatorial concept called the branching datum. This datum encodes the ramification data of cyclic coverings with a fixed genus at each level, generalizing the notion of ``partitions of conductors'' used by Pries and Zhu while preserving its desirable properties. In Section \ref{secmodulispace}, we construct the moduli space $\ASW_{\underline{\iota}}$ and introduce its $p$-rank strata $\ASW_{\underline{\iota},\underline{s}}$. Additionally, we identify the irreducible components of $\ASW_{\underline{\iota},\underline{s}}$. The deformation of elements in $\ASW_{\underline{\iota}}$ is discussed in Section \ref{secdeformation}, where we employ this information to determine the irreducibility of $\ASW_{\underline{\iota}}$. Finally, in Section \ref{secconnectedness}, we explore the connections between the moduli spaces and derive various results concerning their connectedness.

\subsection{Acknowledgements}
The first author would like to express gratitude to the Vietnam Institute for Advanced Study in Mathematics and the National Center for Theoretical Sciences for providing exceptional working conditions. This work is supported by the National Center for Theoretical Sciences. Furthermore, the authors wish to express their appreciation to Joe Kramer-Miller, Andrew Kobin, and Rachel Pries for their invaluable contributions to the discussions. We are particularly grateful to Andrew Obus and J\k{e}drzej Garnek for pointing out inaccuracies in earlier versions of this paper.

\section{Partitions and Artin-Schreier-Witt curves}
\label{secpartionASW}

\subsection{Partitions of integer matrix}
\label{secpartitionmatrix}
This section contains technical content. For better understanding, readers are advised to refer to Section \ref{secASWcurves} before proceeding.

Consider a prime number $p$, and let $(d_1, \ldots, d_n)$ be an $n$-tuple satisfying the following conditions:
\begin{enumerate}[label*=\alph*)]
\item $d_1 \geq 1$
\item $d_i \geq pd_{i-1}-p$
\end{enumerate}

We define $\Omega_{(d_1,\ldots,d_n)}$ as the set of \textit{partitions} of the tuple $(d_1,\ldots,d_n)$. These partitions are $r\times n$ matrices $(r \geq 1)$ of the form:
               \[ M=\begin{bmatrix}\label{matrixpartition}
d_{1,1} & d_{1,2} & \ldots & d_{1,n} \\
d_{2,1} & d_{2,2} & \ldots & d_{2,n} \\
\vdots  & \vdots  & \ddots & \vdots  \\
d_{r,1} & d_{r,2} & \ldots & d_{r,n} \\
\end{bmatrix} \in \mathbb{M}_{r \times n}(\mathbb{Z}_{\ge 0}) \]
subject to the following conditions:
\begin{enumerate}[label*=\arabic*.]
\item \label{item1condbranchingdatum} $d_{i,1} \not\equiv 1 \pmod{p}$
\item \label{item2condbranchingdatum} $d_{i,j} \geq pd_{i,j-1}-p+1$
\item \label{item3condbranchingdatum} if $d_{i,j}>pd_{i,j-1}-p+1$, then $d_{i,j}=pd_{i,j-1}-p+1+a_j$, where $a_j$ is coprime to $p$, and 
\item \label{item4condbranchingdatum} $\sum_{i=1}^r d_{i,j}=d_j$.
\end{enumerate}
For each $j=1,\ldots,r$, let $r_j(M)$ denote the $j$-th row of the matrix $M$. Let $\underline{s}=(s_1, \ldots, s_n) \in \mathbb{Z}_{\ge 0}^n$. We define $\Omega_{\underline{d}, \underline{s}}$ as the subset of $\Omega_{\underline{\iota}}$ consisting of matrices $M=[d_{i,j}]$ where the number of non-zero entries in the $i$-th column is precisely $s_i$. Furthermore, let $M^{p^{n-i}}$ be the sub-$r \times i$ matrix of $M$ that consists of the first $i$ columns.

We introduce a partial ordering denoted by $\prec$ on $\Omega_{(d_1,\ldots,d_n)}$ as follows: Given an $r_1\times n$ matrix $M_1$ and an $r_2\times n$ matrix $M_2$, we say $M_1 \prec M_2$ if each column of $M_2$ is a partition of the corresponding column in $M_1$. To illustrate, consider the case when $p=3$. We can observe the following relation:
\[  \begin{bmatrix}
7 & 21 \\
3 & 8 \\
\end{bmatrix} \prec \begin{bmatrix}
   4 & 10 \\
   3 & 11 \\
   3 & 8\\
   \end{bmatrix}.  \]
Similar to the approach in \cite{MR2985514}, this ordering gives rise to a directed graph denoted as $G_{\underline{\iota}}$. The vertices of the graph correspond to the elements of $\Omega_{\underline{\iota}}$. For any $M, N \in \Omega_{\underline{\iota}}$, there exists an edge from $M$ to $N$ in $G_{\underline{\iota}}$ if $M \prec N$ and there is no other partition between them.

\subsection{Artin-Schreier-Witt curves}
\label{secASWcurves}

Artin-Schreier-Witt theory (see, e.g., \cite{MR1878556}) states that for a field $K$ of characteristic $p$, there is an isomorphism
\[ \cohom^1(G_K, \mathbb{Z}/p^n) \cong W_n(K)/\wp (W_n(K)),\] 
where $G_K$ is the absolute Galois group of $K$, $W_n(K)$ is the ring of length-$n$ Witt vectors over $K$, and $\wp(y):=F(y)-y$ is the Artin-Schreier-Witt isogeny. We say that two elements $\underline{f}, \underline{g} \in W_n(K)$ are in the same \textit{Artin-Schreier-Witt class} if there exists $\underline{h} \in W_n(K)$ such that $\underline{f}=\underline{g}+\wp(\underline{h})$. Such a class is denoted by $[\underline{f}]$.

Now, let's consider an arbitrary commutative subgroup $V$ of $W_n(K)$. We define $\wp^{-1}V$ as the set of $\underline{a} \in W_n(\overline{K})$ such that $\wp(\underline{a}) \in V$. Furthermore, we denote the field obtained by adjoining all elements $\underline{a}\in \wp^{-1}V$ to $K$ as $K(\wp^{-1}V)$. The following proposition offers a classification of extensions $E/K$ that are abelian with exponent at most $p^n$.

\begin{proposition}[{\cite[\S 26, Theorem 6]{MR2371763}}]
\label{propabelianpncorrespondence}
For a field $K$ of characteristic $p$, using the above notations, the map
\begin{equation}
\label{eqnASWcorrespondence}
V \mapsto K(\wp^{-1}V)
\end{equation}
establishes a bijection between the set of subgroups $V$ of $W_n(K)$ containing $\wp(W_n(K))$ and the set of subfields $E$ of $\overline{K}$ such that $E/K$ is abelian with exponent at most $p^n$. The groups $G=\gal(K(\wp^{-1}V))$ and $V/\wp W_n(K)$ are in a dual relationship, meaning that each is canonically isomorphic to the group of characters of the other.
\end{proposition}
From there, we can deduce a classification of $\mathbb{Z}/p^n$-extensions of $K$ in terms of Witt vectors.

\begin{corollary}
\label{coraswcycliccorrespondence}
    The map \eqref{eqnASWcorrespondence} induces a bijection between the set of subgroups $V$ containing $\wp( W_n(K))$ of $W_n(K)$ generated by $\underline{a}=(a_1, \ldots, a_n) \in W_n(K)$, where $a_1 \not\in \wp K$, and $\mathbb{Z}/p^n$-extensions of $K$.  
\end{corollary}

Henceforth, we will use the term ``field extension'' interchangeably with ``character'' (in the context of $ \cohom^1(G_K, \mathbb{Z}/p^n)$). Consequently, we can refer to the product of field extensions.

Due to Artin-Schreier-Witt theory and the correspondence between covers of curves and extensions of function fields, we can represent any $\mathbb{Z}/p^n$-cover $\phi_n: Y_n \to \mathbb{P}_k^1$ be an affine equation of the form
\begin{equation}
    \label{ASWeqn}
    \wp(y_1, \ldots, y_n)=(f_1, \ldots, f_n)=\underline{f},
\end{equation}
where $(f_1, \ldots, f_n)$ is an element of the ring $W_n(k(x))$. Moreover, for $1 \leq i < n$, the character $\phi_i:=(\phi_n)^{p^{n-i}}$ corresponds to the unique $\mathbb{Z}/p^i$-subcover $\phi_i: Y_i \to \mathbb{P}_k^1$ of $\phi_n$ given by
\begin{equation}
\label{eqnASWsub}
\wp(y_1, \ldots, y_i)=(f_1, \ldots, f_i).
\end{equation}
We also denote $\phi_n$ by $\mathfrak{K}_{n}(f_1, \ldots, f_n)$. The following proposition is well-known.
\begin{proposition}
\label{propreducedform}
    Every vector $(f_1, \ldots, f_n) \in W_n(k(x))$ belongs to the same Artin-Schreier-Witt class as a reduced vector $\underline{g}=(g_1, \ldots, g_n)$, where none of the non-zero terms in the partial fraction decomposition of $g_i$ has degree divisible by $p$.
\end{proposition}

\begin{proof}
    When $n=1$, any term $a$ of $g_1$ of degree divisible to $p$ is a $p$-power, i.e., $a=b^p$ for some $b \in k(x)$. Hence, $a$ could be killed off by adding $b^p-b$ to $g_1$. When $n>1$, one can employ the classical induction technique for Witt vectors (e.g., \cite[\S 26, Addendum]{MR2371763}). See \cite[Lemma A.2.3]{MR3714509} for a more detailed proof.
\end{proof}

We call such a vector $\underline{g}$ \textit{reduced}. For the rest of this discussion, we assume that $\underline{f}$ is reduced. 

\begin{proposition}
\label{prop1to1ASW}
    There is a one-to-one correspondence between the set of reduced Witt vectors $\underline{a}=(a_1, \ldots, a_n) \in W_n(k(x))$ modulo the natural $(\mathbb{Z}/p^n)^{\times}$-action, where $a_1 \neq 0$, and the set of $\mathbb{Z}/p^n$-covers of $\mathbb{P}^1_k$.
\end{proposition}

\begin{proof}
Recall that, as stated in Corollary \ref{coraswcycliccorrespondence}, every group $\langle \underline{a},\wp(W_n(k(x)))\rangle/ \wp(W_n(k(x))) \cong \mathbb{Z}/p^n$ corresponds to a cyclic-order-$p^n$ cover. Let $V_{p^n} \subset W_n(k(x))/\wp(W_n(k(x)))$ denote the set of order-$p^n$ elements, which can be represented by reduced length-$n$ Witt vectors with non-zero first entries. It is straightforward to verify that the action of $(\mathbb{Z}/p^n)^{\times}$ on $W_n(k(x))$ by multiplication induces a faithful action on $V_{p^n}$. Furthermore, it is worth noting that the orbit of each element generates the same extension of $k(x)$. Therefore, the proposition follows.
\end{proof}

Now, let us examine the ramification of $\phi_n$. Suppose $\mathscr{P}:= \{P_1, \ldots, P_r\}$ is the set of poles of the $f_i$'s. Then $\mathscr{P}$ is also the branch locus of $\phi_n$. Moreover, for each ramified point $Q_j$ above $P_j$, $\phi_n$ induces a degree-$p^{m_j}$ cyclic extension of the complete local ring $\hat{\mathcal{O}}_{Y_n,Q_j}/\hat{\mathcal{O}}_{\mathbb{P}^1,P_j}$. Hence, we can define the ramification filtration of $\phi_n$ at a branch point $P_j$.

Suppose the inertia group of $Q_j$ is $\mathbb{Z}/p^{m_j}$ (where $n\le m_j$). We say that the \textit{$i$-th ramification break} of $\phi_n$ at $P_j$ is $-1$ for $i \le n-m_j$. For $i > n-m_j$, the $i$-th ramification break of $\phi_n$ at $P_j$ is the $(i-n+m_j)$-th one of $\hat{\mathcal{O}}_{Y_n,Q_j}/\hat{\mathcal{O}}_{\mathbb{P}^1,P_j}$. We denote the $i$-th upper ramification break of $\phi_n$ at $P_j$ by $u_{j,i}$, and we define the \textit{$i$-th conductor of $\phi$ at $P_j$} to be $e_{j,i}:=u_{j,i}+1$. One then can explicitly compute the ramification filtration at each branch point of $\phi_n$ in terms of $\underline{f}$ as follows.

\begin{theorem}[{\cite[Theorem 1]{MR1935414}}]
\label{theoremcaljumpirred}
With the assumptions and the notations as above, we have
\begin{equation}
\label{eqnformulalowerjumpasw}
    u_{j,i}=\max\{ p^{i-l} \deg_{(x-P_j)^{-1}} (f_{l}) \mid l=1, \ldots, i\}, 
\end{equation}
for $i>n-m_j$.
\end{theorem}

\begin{proposition}
With the above notations, the genus of $Y_i$ is
    \begin{equation}
    \label{eqngenera}
        g_{Y_i} = 1+ p^i(g_X-1)+ \frac{\sum_{l=1}^i (\sum_{j=1}^r e_{j,l})(p^l-p^{l-1})}{2}
    \end{equation}
\end{proposition}

\begin{proof}
Applying the Grothendieck-Ogg-Shafarevich (GOS) formula \cite[Exp. X formula 7.2]{MR0491704} to $\phi_i: Y_i \xrightarrow{} X$, we obtain
\begin{equation}
\label{eqnGOS}
2g_{Y_i}-2=\deg(\phi_i) (2g_X-2) + \sum_{j=1}^r \sw_{P_j}(\phi_i),
\end{equation}
where $\sw_{P_j}(\phi_i)$ is the Swan conductor of $\phi_i$ at $P_j$. See also \cite[Theorem 4.2.9]{MR2415398} for a generalized version of the formula to higher dimension 
According to \cite[Fact 2.3]{MR3194816}, the value of $\sw_{P_j}(\phi_i)$ (which Pop calls the degree of the different) is determined by the ramification at $P_j$
\begin{equation}
\label{eqndifferent}
\sw_{P_j}(\phi_i)=\sum_{l=1}^i e_{j,l}(p^l-p^{l-1}).
\end{equation}
This immediately implies the claim about the genus of $Y_i$.
\end{proof}

\begin{remark}
\label{remarkgosramified}
    An alternative derivation of equation (\ref{eqngenera}) can be obtained by considering the ramified points. Let $P_j \in \mathscr{P}$ be a ramified point, and let $m^i_j$ be the inertia degree of $\phi_i$ at $P_j$. The number of ramified points in $Y_i$ lying over $P_j$ is then $p^{i-m^i_j}$, and we label these points as $Q_{j,1}, \ldots, Q_{j,p^{i-m^i_j}}$. Below is another version of (\ref{eqnGOS}), based on the ramified points, which is known as the Riemann-Hurwitz formula \cite{MR0463157}
    \begin{equation}
    \label{eqnRH}
    \begin{split}
        2g_{Y_i}-2 & =\deg(\phi_i) (2g_X-2) + \sum_{j=1}^r \sum_{l=1}^{p^{i-m^i_j}} \deg\big(\mathscr{D}_{Q_{j,l}}\big). \\
        & =p^i (2g_X-2) + \sum_{j=1}^r p^{i-m^i_j} \deg\big(\mathscr{D}_{Q_{j,1}}\big),
    \end{split}
    \end{equation}
    where $\mathscr{D}_{Q_{j,1}}$ is the different at $Q_{j,i}$ in the classical sense \cite[III, \S 3]{MR554237}. Specifically, we have
    \begin{equation*}
        \deg\big(\mathscr{D}_{Q_{j,1}}\big)= \sum_{s=i-m^i_j+1}^i e_{j,s}(p^{s-i+m^i_j}-p^{s-i+m^i_j-1}).
    \end{equation*}
    Therefore, we obtain $\sw_{P_j}(\phi_i)=\sum_{j=1}^{p^{i-m^i_j}} \deg\big(\mathscr{D}_{Q_{j,l}}\big)=\sum_{l=1}^i e_{j,l}(p^l-p^{l-1})$ and thus (\ref{eqngenera}). For more detailed computation, refer to \cite[\S 2.A.]{MR3194816}.
\end{remark}

We can deduce from (\ref{eqngenera}) that $\mathbb{Z}/p^n$-covers with the same genus on each sub-cover must have the same $d_i: = \sum_{j=1}^r e_{j,i}$. We use an $r \times n$ matrix to record the branching datum of the cover, as shown below:
\begin{equation}
\label{eqnmatrixbranchingdatum}    
\begin{bmatrix}
e_{1,1} & e_{1,2} & \ldots & e_{1,n} \\
e_{2,1} & e_{2,2} & \ldots & e_{2,n} \\
\vdots  & \vdots  & \ddots & \vdots  \\
e_{r,1} & e_{r,2} & \ldots & e_{r,n} \\
\end{bmatrix}
\end{equation}
We refer to this matrix as the \textit{branching datum} of $\phi_n$ and the vector $\underline{d}:= (d_1, \ldots, d_n)$ as its \textit{conductors}. 

From now on, we will utilize the notion of partitions of integer matrices introduced in \S \ref{secpartitionmatrix}. 

\begin{proposition}
\label{propcorrespondencematrixcover}
A matrix of the form (\ref{eqnmatrixbranchingdatum}) is a branching datum of a $\mathbb{Z}/p^n$-cover if and only if it belongs to $\Omega_{\underline{d}}$.
\end{proposition}

\begin{proof}
    ``$\Rightarrow$'': By Theorem \ref{theoremcaljumpirred}, a matrix of the form (\ref{eqnmatrixbranchingdatum}) must satisfy conditions \ref{item1condbranchingdatum}, \ref{item2condbranchingdatum}, and \ref{item3condbranchingdatum} in order to be a branching datum of a $\mathbb{Z}/p^n$-cover.

    ``$\Leftarrow$'': Suppose we are given $M=(e_{j,i}) \in \Omega_{\underline{d}} \cap M_{r \times n}(\mathbb{Z}_{\ge 0})$. Let $P_1, \ldots, P_r$ be $r$ distinct points in $\mathbb{P}^1_k$. For each $1 \leq i \leq n$, set
\begin{equation*}
    u_{j,i} := e_{j,i}-1 \quad \text{and} \quad c_{j,i}:=\begin{cases}
    1 & \text{if $e_{j,i} \not\equiv 1 \pmod{p}$ and $e_{j,i} \neq -1$}  \\
    0 & \text{otherwise} \\
    \end{cases}.
\end{equation*}
Define a cover $\phi_n: Y_n \xrightarrow{} \mathbb{P}^1_k$ by
\begin{equation*}
    \wp(y_1, \ldots, y_n)=\left( \sum_{j=1}^r  \frac{c_{j,1}}{(x-P_j)^{u_{j,1}}}, \ldots, \sum_{j=1}^r  \frac{c_{j,n}}{(x-P_j)^{u_{j,n}}} \right) \in W_n(k(x)).
\end{equation*}
It can be verified that this cover has the branching datum $M$.
\end{proof}

\begin{remark}
If $\phi_n$ has the branching datum \eqref{eqnmatrixbranchingdatum}, then the branching datum of $\phi_i:=(\phi_n)^{p^{n-i}}$ is given by the $r \times i$ matrix formed by taking the first $i$ columns of \eqref{eqnmatrixbranchingdatum}.
\end{remark}

\begin{example}
    \label{examplesameendsdifferent}
Consider two $\mathbb{Z}/9$-covers $f: Y_2 \rightarrow Y_1 \rightarrow \mathbb{P}^1$ and $g: Z_2 \rightarrow Z_1 \rightarrow \mathbb{P}^1$ in characteristic 3 with branching data $M$ and $N$, respectively:
\begin{equation*}
M=\begin{bmatrix}
2 & 4 \\
3 & 7 \\
\end{bmatrix} \hspace{10mm} N=\begin{bmatrix}
2 & 6 \\
0 & 6 \\
\end{bmatrix}.
\end{equation*}
For instant, $f=\mathfrak{K}_{2}\big(x+\frac{1}{x^2}, 0\big)$ and $g=\mathfrak{K}_2\big(x, x^5+\frac{1}{x^5} \big)$. It can be shown that $g_{Y_2}=g_{Z_2}=30$, while $g_{Y_1}=3$ and $g_{Z_1}=0$.
\end{example}

\begin{remark}
    Since any flat family of $\mathbb{Z}/p^n$-covers must have $\mathbb{Z}/p^i$-sub-covers with the same genera, it is reasonable to utilize the conductors $(d_1, \ldots, d_n)$ as an invariant for the family.
\end{remark}


\begin{definition}
With the notation above, we call the Weil divisor $D_i=\sum_{l=1}^r d_{l,i} P_l$ the \textit{branched divisor} of $\phi_n$'s $\mathbb{Z}/p^i$-sub-cover.
\end{definition}

\begin{example}
Suppose $\chi$ is a $\mathbb{Z}/5^2$-cover of $\mathbb{P}^1_k$ defined by the following equation
\[ \wp(\underline{y})= \bigg( \frac{1}{x}+\frac{1}{(x-1)}, \frac{1}{(x-1)^7}+\frac{1}{(x-2)^{12}} \bigg)=:\underline{a}. \]
As each entry of the Witt vector has no terms of degree divisible by $5$, it is reduced. Furthermore, apply \cite[Theorem 1]{MR1935414}, we learn that $\chi$ branches at three points $x=0$, $x=1$, and $x=2$ with conductors $(2, 6)$, $(2, 8)$, and $(0, 13)$, respectively. Hence, its branching datum is
\[ \begin{bmatrix}
2 & 6 \\
2 & 8 \\
0 & 13 \\
\end{bmatrix}. \]
The $\mathbb{Z}/5$-subcovering is the Artin-Schreier curve represented by $y_1^5-y_1=\frac{1}{x}+\frac{1}{x-1}$ with branching datum $[2,2]^{\top}$.
\end{example}

\subsection{The \texorpdfstring{$p$}{p}-rank of Artin-Schreier-Witt curves and partitions}   
\label{secmatrixASW}

Recall that the $p$-rank of a curve $C/k$, denoted by $p_C$, is defined as the integer such that the cardinality of $\text{Jac}(C)[p](\overline{k})$ is $p^{p_C}$. This birational invariant holds significant importance as it has been shown to influence the automorphism group $\text{Aut}(C)$ \cite{MR902787} and the Newton polygon of $C$'s associated $L$-function \cite{MR3287680}.

When a cover $\rho: X \xrightarrow{} Y$ is $G$-Galois, where $G$ is a $p$-group, the Deuring-Shafarevich formula  \cite[Corollary 1.8]{MR742696} indicates that
\begin{equation}
    \label{eqndeuringshafarevich}
    1-p_X= \lvert G \rvert (1-p_Y)-\sum_{x \in X^{\ram}} (e_x-1),
\end{equation}
where $X^{\ram}$ represents the ramified locus of $\rho$, and $e_x$ is the ramification index at the point $x$. We now apply this formula to the case where $G \cong \mathbb{Z}/p^n$, $Y \cong \mathbb{P}^1_k = \Proj k[x,u]$, and $\rho$ is defined by a reduced $\underline{f}=(f_1, \ldots, f_n) \in W_n(k(x))$. This leads to the following expression:
\[ p_X=1-p^n+\sum_{x_i \in X^{\ram}}(e_{x_i}-1), \]
where $e_{x_i}=p^{m_i}$ with $m_i:=\min \{j \mid \nu_{P_i}(f_j) <0 \}$. Suppose in addition that $\rho$ has branching datum $M$ as in \eqref{eqnmatrixbranchingdatum}. Then one can apply the analysis on Remark \ref{remarkgosramified} to get
\begin{equation}
    \label{eqnprank}
    p_X=1-p^n+\sum_{i=1}^{n} p^{n-i}(p^i-1)s_i. 
\end{equation}
where $s_i$ ($1 \le i \le n$) is the number of branch points of ramification index $p^i$. Thus, the $p$-rank of a cover is determined by $\underline{s}=(s_1, \ldots, s_n)$. 
In addition, a straight-forward computation shows that, if one fixed the $p$-ranks of the sub-covers of a family of Artin-Schreier-Witt covers, then $\underline{s}$ is also fixed. The $n$-Artin-Schreier-Witt curves of conductors $\underline{\iota}$ and $p$-ranks corresponding to $\underline{s}$ are intimately related to the partition set $\Omega_{\underline{\iota}, \underline{s}}$ as defined in \S \ref{secpartitionmatrix}.

\begin{lemma}
There exists an Artin-Schreier-Witt $k$-curve with conductors $\underline{\iota}$ and $p$-ranks $\underline{s}$ if and only if $\Omega_{\underline{\iota}, \underline{s}}$ is nonempty.
\end{lemma}

\begin{proof}
``$ \Rightarrow$'' Suppose $\phi_n$ is an Artin-Schreier-Witt $k$-curve with conductors $\underline{\iota}$ and $p$-ranks $\underline{s}$. By Theorem \ref{theoremcaljumpirred}, the associated matrix of $\phi_n$ belongs to $\Omega_{\underline{\iota}, \underline{s}}$, implying that $\Omega_{\underline{\iota}, \underline{s}}$ is nonempty.

``$ \Leftarrow$'' Let $M$ be a matrix in $\Omega_{\underline{\iota}, \underline{s}}$, we can apply the process in the proof of Proposition \ref{propcorrespondencematrixcover} to construct a curve with datum $M$.
\end{proof}

\section{Moduli space of Artin-Schreier-Witt curves}
\label{secmodulispace}

Consider fixed parameters $p, \underline{\iota}=(\iota_1, \ldots, \iota_n)$ and a fixed algebraically closed ground field $k$ of characterisic $p$. In this section, we study the strata $\mathcal{ASW}_{\underline{\iota},M}$ of the moduli space $\mathcal{ASW}_{\underline{\iota}}$ where $M$ is a matrix of the form (\ref{matrixpartition}). We show the irreducible components of $\mathcal{ASW}_{\underline{\iota}}$ are in bijection with the elements of $\Omega_{\underline{\iota}}$ and find the dimensions of these components. In the following, $\mathcal{M}_{g}$ denotes the moduli space of smooth algebraic curves of genus $g$ over $k$.

\subsection{Artin-Schreier-Witt covers} 

Let $S$ be a $k$-scheme. An $S$-curve is a smooth, proper, surjective, finitely presented morphism $C \xrightarrow{} S$ whose geometric fibers are smooth, projected connected curves. Similar as in  \cite{MR2985514}, we define an \textit{Artin-Schreier-Witt curve} $C$ over $S$ to be an $S$-curve for which there exists an (unspecified) inclusion $j: \mathbb{Z}/p^n \to Aut_S(C)$. By tracking different data associated to Artin-Schreier-Witt curves, we can define three different Moduli functors:
 \begin{definition}
 Fix the data of conductors $\underline{\iota}$. We define the following Moduli functors from the category of  $k$-schemes to the category of sets:
     \begin{enumerate}
         \item The functor $\mathcal{ASW}_{\underline{\iota}} \subset \mathcal{M}_g$ maps
         \begin{align}
              S & \mapsto \{ C/S \}/_{\sim}
         \end{align}
         where $C$ is a $S$-curve, such that there exists an inclusion $\mathbb{Z}/p^n \to Aut(C/S)$, $C/S$ is isomorphic to $\mathbb{P}^1_S$ and each geometric fiber of the associated quotient map $C \to C/(\mathbb{Z}/p^n)$ is a $\mathbb{Z}/p^n$-cover with conductors $\underline{\iota}$.
         \item The functor $\mathcal{ASW}act_{\underline{\iota}}$ maps
         \begin{align}
              S & \mapsto \{ C/S, \kappa:  \mathbb{Z}/p^n \to Aut(C/S) \}/_{\sim}
         \end{align}
         with $C$ as above, but here the inclusion $\kappa$ is part of the data.
         \item The functor $\mathcal{ASW}cov_{\underline{\iota}}$ maps
         \begin{align}
              S & \mapsto \{ C/S, \kappa:  \mathbb{Z}/p^n \to Aut(C/S), \psi: C/(\mathbb{Z}/p^n) \to \mathbb{P}^1_S \}/_{\sim}
         \end{align}
         with $C, \kappa$ as above and a fixed isomorphism $\psi:  C/(\mathbb{Z}/p^n) \to \mathbb{P}^1_S$.
     \end{enumerate}
 \end{definition}

The moduli functors $\mathcal{ASW}_{\underline{\iota}}$ and $\mathcal{ASW}cov_{\underline{\iota}}$ are those of intrinsic interest, while one can think about $\mathcal{ASW}act_{\underline{\iota}}$ as an auxiliary object sitting in between them.
\begin{remark}
As for each $0 \geq i \geq n$ there is a unique $\mathbb{Z}/p^i$ subgroup in $\mathbb{Z}/p^n$, so we can factor the map $C \to C/\mathbb{Z}/p^n$ uniquely into a sequence of $\mathbb{Z}/p$-Galois covers. With this observation, we can rephrase $\mathcal{ASW}cov_{\underline{\iota}}$ as:
\begin{equation}
    \begin{split}
        \mathcal{ASW}\cov_{\overline{\iota}}: \{ k\text{-Schemes} \} & \xrightarrow{} \text{Sets} \\
        S & \mapsto \{ \phi: C_n \xrightarrow{} C_{n-1} \xrightarrow{} \ldots \xrightarrow{} C_0 \cong \mathbb{P}^1_S \}/_{\sim}
    \end{split}
\end{equation}
where $\phi$ is $\mathbb{Z}/p^n$-Galois, hence $C_i \rightarrow C_0$ is $\mathbb{Z}/p^i$-Galois. The curve $C_i$ can be identified as the $\mathbb{Z}/p^{n-i}$-invariant part of $C_n$.

A direct computation with the Riemann-Hurwitz formula shows that with this notation, it holds for the genera of the subquotients 
\begin{align}
g_{C_i}=1+ p^i(g_{C_0}-1)+\frac{\sum_{l=1}^i \iota_{l}(p^j-p^{j-1})}{2}. \label{gen-form}
\end{align}
\end{remark}

\begin{remark}
    We restrict ourselves to the case where the genus of $C_1$ is at least $2$ because of technical reasons. This is necessary to avoid the existence of infinitesimal automorphisms on subquotients and could be avoided by adding markings to the curve. However, these fringe cases are not really relevant for our theory, so they are ignored at this point.
\end{remark}

\begin{proposition}
There is an algebraic stack representing $\mathcal{ASW}_{\underline{\iota}}$.
\end{proposition}

\begin{proof}
We start by showing the representability of $\mathcal{ASW}act_{\underline{\iota}}$.
Let $G$ be a finite group and $\MggpG$ be the 
Moduli functor from the category of $k$-schemes parametrising smooth algebraic curves $C$ of genus $g$ together with a faithful $G$-action on $C$ such that the quotient $C/G$ is a smooth curve of genus $g'$. From \cite[Proposition 2.7]{MR2223481} it follows that for any finite group $G$,  $\MggpG$ is an algebraic stack as an open substack of an algebraic stack.

By considering the $\mathbb{Z}/p^{i}$-invariant part of $C$ for each $1 \leq i \leq n$,
we can factor an object of $\Mgg[\mathbb{Z}/p^n]$ uniquely as a sequence 
\begin{center}
		\begin{tikzcd}
		C_n \arrow[r, "f_n"] & C_{n-1} \arrow[r, "f_{n-1}"] & \dots \arrow[r, "f_1"]  & C_0
		\end{tikzcd}
\end{center}
with each $f_i$ being a $\mathbb{Z}/p$-cover. If we restrict this construction to $\mathcal{ASW}_{\underline{\iota}}$, then the genera $g_i$ of each curves $C_i$ is determined by $\iota_i$, 
as they can be computed from the data $\underline{\iota}$ by \ref{gen-form}. The locus of objects of $\Mgg[\mathbb{Z}/p^n]$ with $\mathbb{Z}/p^i$ subquotients having genus $g_i$ form a closed substack that we will denote by $\mathcal{M}_{g=g_n, g_{n-1}, \dots, g'=g_0}[\mathbb{Z}/p^n]$. 

Note that for each $i=0, \dots, n$ by 
quotienting out the $\mathbb{Z}/p^i$-action we get a morphism 
\begin{align*}
    \pi_i: \Mgg[\mathbb{Z}/p^n] \to \coprod_{j=0}^g \mathcal{M}_{j,g',k}[\mathbb{Z}/p^{n-i}]
\end{align*}
which by \cite[Proposition 2.4]{MR2223481} is representable. We can write
\begin{align*}
    \mathcal{M}_{g=g_n, g_{n-1}, \dots, g'=g_0}[\mathbb{Z}/p^n] = \bigcap_{i=0}^n \pi_i^{-1} (\mathcal{M}_{g_i,g',k}[\mathbb{Z}/p^{n-i}])
\end{align*}
showing that $\mathcal{M}_{g=g_n, g_{n-1}, \dots, g'=g_0}[\mathbb{Z}/p^n]$ is an open and closed substack of $\Mgg[\mathbb{Z}/p^n]$. Now the functors $\mathcal{ASW}act_{\underline{\iota}}$ and $\mathcal{M}_{g=g_n, g_{n-1}, \dots, g'=g_0}[\mathbb{Z}/p^n]$ embed into
 \begin{align*}
     \mathcal{F}:={\mathcal{M}}_{g_n,g_{n-1},k}[\mathbb{Z}/p] \times_{\mathcal{M}_{g_{n-1}}} {\mathcal{M}}_{g_{n-1},g_{n-2},k}[\mathbb{Z}/p] \times_{\mathcal{M}_{g_{n-2}}} \dots \times_{\mathcal{M}_{g_{1}}} {\mathcal{M}}_{g_{1},g_{0},k}[\mathbb{Z}/p]
 \end{align*}
We know that the image of $\mathcal{M}_{g=g_n, g_{n-1}, \dots, g'=g_0}[\mathbb{Z}/p^n]$ in $\mathcal{F}$ is an algebraic substack. Furthermore, the image of $\mathcal{ASW}act_{\underline{\iota}}$ consists of those sequences in $\mathcal{M}_{g=g_n, g_{n-1}, \dots, g'=g_0}[\mathbb{Z}/p^n]$, such that the conductors of the compositions $f_1 \circ \dots \circ f_n$
are given by $\underline{\iota}$. As the conductors are given by algebraic conditions, we can identify $\mathcal{ASW}act_{\underline{\iota}}$ as a closed substack of $\mathcal{M}_{g=g_n, g_{n-1}, \dots, g'=g_0}[\mathbb{Z}/p^n]$ 
proving the representability of $\mathcal{ASW}act_{\underline{\iota}}$ by an algebraic stack.

From the proof of 1.4. from \cite{MR1216034} we know that the forgetful morphism $\mathcal{M}_g[\mathbb{Z}/p^n] \to \mathcal{M}_g$ is representable and quasi-finite, so closed. We can identify $\mathcal{ASW}_{\underline{\iota}}$ as the image under this morphism, so it is a closed substack of $\mathcal{M}_g$ and it follows the statement.
\end{proof}
\begin{corollary}
There is an algebraic stack representing $\mathcal{ASW}\cov_{\underline{\iota}}$.
\end{corollary}

\begin{proof}
    Let $\mathcal{U}\mathcal{ASW}_{\underline{\iota}}$ be the universal Artin-Schreier-Witt-curve which exists because of Yoneda's Lemma for stacks.
    Now consider an (arbitrary) choice of three disjoint sections $\sigma_1, \sigma_2, \sigma_3: \mathcal{ASW}_{\underline{\iota}} \to \mathcal{U}\mathcal{ASW}_{\underline{\iota}}$ such that for every $\mathbb{Z}/p^n$-action in $\mathcal{ASW}act_{\underline{\iota}}$, the sections are mapping to different orbits. This is possible, as there are only finitely many such actions on $\mathcal{U}\mathcal{ASW}_{\underline{\iota}}$. From Theorem $3$ from \cite{HALL2014194} it follows the existence of the algebraic stack
    \begin{align*}
        \mathcal{H}:= \mathcal{HOM}_{\mathcal{ASW}_{\underline{\iota}}}(\mathcal{U}\mathcal{ASW}_{\underline{\iota}}, \mathbb{P}^1_{\mathcal{ASW}act_{\underline{\iota}}})
    \end{align*}
    parametrising morphisms from Artin-Schreier-Witt-curves to $\mathbb{P}^1$. The locus of morphisms mapping the $\sigma_i$ to pairwise distinct
    points forms an open substack. Now, such a morphism $f$ over a scheme $S$ is an object of $\mathcal{ASW}\cov_{\underline{\iota}}$ if and only if there exists an isomorphism $\psi: \mathbb{P}^1_S \to C_S/(\mathbb{Z}/p^n)$ and an $\mathbb{Z}/p^n$-action in $\mathcal{ASW}act_{\underline{\iota}}$ making

    \begin{center}
		\begin{tikzcd}
		& C_S \arrow[dr, "f"] \arrow[dl, "/ (\mathbb{Z}/p^n)"] & \\
        C_S/ (\mathbb{Z}/p^n) \arrow[rr, "\psi"]& & \mathbb{P}^1_S
		\end{tikzcd}
\end{center}
commute. As we keep track of three markings, there is exactly one possible choice for such a $\psi$. Furthermore, there are only finitely many such $\mathbb{Z}/p^n$-actions, so this condition exhibits $\mathcal{ASW}\cov_{\underline{\iota}}$ as a locally closed substack of  $\mathcal{H}$, which proves 
its representability by an algebraic stack.
\end{proof}

\begin{lemma}
\label{lemmacoverandcurve}
Let $g \ge 2$. Then there is a morphism $F: \mathcal{ASW} \cov_{\underline{\iota}} \xrightarrow{} \mathcal{ASW}_{\underline{\iota}}$ and the fiber of $F$ over every geometric point of $\mathcal{ASW}_{\underline{\iota}}$ has dimension $3$.
\end{lemma}

\begin{proof}
The proof is similar to one for \cite[Lemma 3.1]{MR2985514}.
\end{proof}

\subsection{The branch divisor} 

Suppose $\underline{\iota}=(\iota_1, \ldots, \iota_n)$. We define the contravariant functor $$\C\divisor_{\underline{\iota}}: \{  k\text{-Schemes} \} \rightarrow {\text{Sets} },$$which associates to each scheme $S$ the set of isomorphism classes of $n$ distinct relative Cartier divisors $D_1, D_2, \ldots, \allowbreak D_n$ of $\mathbb{P}^1_S$ such that the degree of $D_i$ is $\iota_i$. This functor is representable as it is the product of the representable functors $\C \divisor_{\iota_i}$ \cite{MR2985514}.

Suppose $S=\spec K$ is a field of characteristic $p$. Given $\underline{D}=(D_1, \ldots, D_n) \in \mathcal{C} \divisor_{\underline{\iota}}(S)$, one can associate with $\underline{D}$ a set of $n$ locally principal effective Weil divisors of $\mathbb{P}^1_S$. More precisely, after a finite flat extension $S' \rightarrow S$, we may write $D_i'=D_i \times_S S'$ as $\sum_{l=1}^{r} e_{l,i} P_l$, where $\sum_{l=1}^r e_{l,i} = \iota_i$, and ${P_1, \ldots, P_r }$ is a set of distinct horizontal sections of $\mathbb{P}^1_{S'}$. Let $M(\underline{D})$ be the $r \times n$ matrix $(e_{l,m})$, where $r_1(M) \le r_2(M) \le \ldots \le r_r(M)$. The partition $M$ of $\underline{\iota}$ induces a natural stratification $\C\divisor_{\underline{\iota}, M}$ of $\C\divisor_{\underline{\iota}}$ by varying the sections ${P_1, \ldots, P_r}$ associated with $\underline{D}$.

We set $\mathcal{C} \divisor_{\underline{\iota}}^1:=\bigcup_{M \in \Omega_{\underline{\iota}}} \mathcal{C} \divisor_{\underline{\iota},M}$. For a fixed $M$, let $H_M \subset S_{r}$ be the subgroup of the symmetric group generated by all transpositions $(j_1, j_2) \in \mathbb{N} \times \mathbb{N}$ for which $r_{j_1}(M)=r_{j_2}(M)$, i.e., the transpositions that interchange identical rows of $M$.

\begin{lemma}
\label{lemmadimstratamatrix}
If $M \in \Omega_{\underline{\iota}}$ is an $r \times n$ matrix, then $\mathcal{C}\divisor_{\underline{\iota},M}$ is irreducible of dimension $r$.
\end{lemma}

\begin{proof}
Let $\Delta$ denote the weak diagonal of $(\mathbb{P}^1)^{r}$, consisting of $r$-tuples with at least two coordinates equal. It is straightforward to see that the quotient of $(\mathbb{P}^1)^{r} \setminus \Delta$ by the natural action of $H_M$ is an irreducible scheme of dimension $r$. Moreover, the spaces $\mathcal{C} \divisor_{\underline{\iota}}$ and $[(\mathbb{P}^1)^{r} \setminus \Delta]/H_M$ are locally isomorphic for the finite flat topology. In particular, the isomorphism identifies each $\underline{D}=(D_1, \ldots, D_n)$ in $\mathcal{C}\divisor_{\underline{\iota},M}$ with the equivalence class of $(P_1, \ldots, P_r)$ in $[(\mathbb{P}^1)^{r} \setminus \Delta]/H_M$, where $P_i$ is the $i$-th point in the set ${P_1, \ldots, P_r }$ of distinct horizontal sections of $\mathbb{P}^1_{S'}$ associated to $\underline{D}$. Thus, $\mathcal{C}\divisor_{\underline{\iota},M}$ is an irreducible scheme of dimension $r$.
\end{proof}

\begin{proposition}
\label{propthemapb}
There is morphism $B: \ASW\cov_{\underline{\iota}} \xrightarrow{} \mathcal{C}\divisor_{\underline{\iota}}$ and the image of $B$ is $\mathcal{C}\divisor^1_{\underline{\iota}}$.
\end{proposition}

\begin{proof}
    Let $\phi_n$ be a $\mathbb{Z}/p^n$-cover over $S$, and for each $i = 1, \ldots, n-1$, let $\phi_i$ be the $\mathbb{Z}/p^i$-subcover of $\phi_n$. Define $D_i$ as the closed subscheme of branched points of $\phi_i$, which is a relative Cartier divisor of $\mathbb{P}^1_S$. Let $\underline{D} = (D_1, \ldots, D_n)$. The transformation $\ASW\cov_{\underline{\iota}} (S) \rightarrow \mathcal{C} \divisor_{\underline{\iota}, M}(S)$, given by $\phi_n \mapsto \underline{D}$, is functorial, and it yields a morphism $B: \ASW\cov_{\underline{\iota}} \rightarrow \mathcal{C} \divisor_{\underline{\iota}}$ by Yoneda's Lemma.

For each $\underline{D} \in \mathcal{C} \divisor_{\underline{\iota}}(S)$, let us consider the restriction of $B$ on $M(D)$. As discussed at the beginning of the section, one can identify a pullback $D'_i :=D_i \times_S S'$ with an effective Weil divisor $\sum_{l=1}^r e_{l, i} P_l$, where ${ P_1, \ldots, P_r }$ is a set of distinct horizontal sections of $\mathbb{P}^1_{S'}$. If $\underline{D}=B(\phi_n)$ for some $\phi_n \in \ASW\cov_{\underline{\iota}}(S)$, then ${P_1, \ldots, P_s}$ constitutes the branch locus of the pullback $\phi_n'=\phi_n \times_S S'$ and $d_{l,i}=e_{l,i}-1$ is the $i$-th upper jump of $\phi_n$ above the geometric generic point of $P_l$. In addition, the $e_{l,i}$'s have to follow criterion \ref{item1condbranchingdatum}, \ref{item2condbranchingdatum}, and \ref{item3condbranchingdatum} by Proposition \ref{propcorrespondencematrixcover}. Thus, the image of $B$ is contained in $\mathcal{C} \divisor^1_{\underline{\iota}}$.

Suppose $M=(e_{j,i}) \in \Omega_{\underline{\iota}}$ is an $r \times n$ matrix. To prove that $\mathcal{C} \divisor^1_{\underline{\iota}, M}$ is contained in the image of $B$, it suffices to work locally in the finite flat topology. Given $\underline{D}=(D_1, \ldots, D_n)$ where $D_i =\sum_{l=1}^r e_{l,i} P_l$, define $D'_i :=\sum_{l=1}^s e'_{l,i} P_l$ where
\[e'_{l,i}= \begin{cases} 
      e_{l,i}-1 & e_{l,i} > 1 \text{ and } e_{l,i} \not\equiv 1 \pmod{p} \\
      0 &\text{otherwise}  \\
   \end{cases}.
\]
Then, we choose a (possibly constant) function $f_i \in \mathcal{O}(S)(x)$ with $\divisor_{\infty}(f_i(x))=D'_i$. It is easy to see that the cover $\phi_n$ defined by $\wp(\underline{y})=(f_1(x), \ldots, f_n(x))$ has branching datum $M$. Thus, $\underline{D} \in \Image(B)$ and $\mathcal{C} \divisor^1_{\underline{\iota}, M}$ is contained in the image of $B$ as desired.
\end{proof}

We use $\ASW_{\underline{\iota},M}$ (respectively, $\ASW\cov_{\underline{\iota},M}$) to denote the locally closed reduced subspace of $\ASW_{\underline{\iota}}$ (respectively, $\ASW\cov_{\underline{\iota}}$) whose geometric points correspond to Artin-Schreier-Witt curves (respectively, covers) with branching data $M$. The morphisms $F$ and $B$ preserve the partition $M$. We define $F_M:\ASW\cov_{\underline{\iota},M} \rightarrow \ASW_{\underline{\iota},M}$ and $B_M:\ASW\cov_{\underline{\iota},M} \rightarrow C\divisor^1_{\underline{\iota},M}$ as the natural restrictions.

\subsection{Artin-Schreier-Witt covers with fixed branching divisor}
In this section, we study the $p$-rank strata $\ASW_{\underline{\iota}, \underline{s}}$ of the moduli space $\ASW_{\underline{\iota}}$. Fix an $r \times n$ matrix $M=(\iota_{j,i}) \in \Omega_{\underline{\iota}}$ and a family of divisors $\underline{D} \in \mathcal{C} \divisor_{\underline{\iota},M}$. Set
\[ s(M)=(s_1, \ldots, s_n)=: \underline{s} \in \mathbb{Z}_{\ge 0}^n, \]
where $s_i$ is the number of nonzero entries of the $i$-th column of $M$, hence $s_n=r$.

For $1 \le j \le r$ and $1 \le i \le n$, let $t_{j,i}:=d_{j,i}-\lfloor d_{j,i}/p \rfloor$ (resp. $t_{j,i}=0$) where $d_{j,i}:=\iota_{j,i}-1$ when $\iota_{j,i} \ge 1$ (resp. $\iota_{j,i}=0$). Let $N_M=\sum_j \sum_i t_{j,i}$. For a fixed $1 \le j \le r$, let $L_{j,i}:=(\mathbb{A}^1)^{t_{j,i}-1} \times (\mathbb{A}^1 \setminus \{ 0\})$ for $\iota_{j,i} \not\equiv 1 \pmod {p}$ and $L_{j,i}:=(\mathbb{A}^1)^{t_{j,i}}$ for $\iota_{j,i} \equiv 1 \pmod{p}$. Let $L:= \times_{j,i} L_{j,i}$. Then there is an obvious action on $L$ by the subgroup $H_M \subset S_r$. Set $L_{\underline{D}}:=L/H_M$.

\begin{proposition}
\label{propfiberofBM}
The fiber $\ASW\cov_{\underline{\iota},\underline{D}}$ of $B_M$ over $\underline{D}$  is locally isomorphic for the finite flat topology to a quotient of $L_{\underline{D}}$ by a finite group. Thus,  $\ASW\cov_{\underline{\iota},\underline{D}}$ is irreducible with dimension $N_M$ over $k$. 
\end{proposition}

\begin{proof}
Let $\eta$ denote a labeling of the $r$ points in the support of $\underline{D}$. Let $\ASW\cov^{\eta}_{\underline{\iota},\underline{D}}$ be the contravariant functor which associates to $S$ the set of coverings $\phi$ in the fiber $\ASW\cov_{\underline{\iota},\underline{D}}(S)$ along with a labeling $\eta$ of the branch locus. To prove the proposition, it suffices to show that the moduli space for $\ASW\cov^{\eta}_{\underline{\iota},\underline{D}}$ is locally isomorphic to a quotient of $L$.

Suppose we have a point $\phi_n: Y_n \xrightarrow{} \mathbb{P}^1_S$ in $\ASW\cov^{\eta}_{\underline{\iota},\underline{D}}(S)$ represented by a reduced Witt vector $\underline{f}=(f_1, \ldots, f_n) \in W_n(S)$. After a finite flat extension $S' \xrightarrow{} S$, we can assume that $\underline{f}$ is in reduced form. Let $\underline{D}$ have support ${P_1, \ldots, P_r}$. We can express each $f_i$, $1 \le i \le n$, as
\begin{equation*}
    f_i =\sum_{j=1}^r \sum_{l=1}^{d_{j,i}} \frac{c^l_{j,i}}{(x-P_j)^l}
\end{equation*}
where $d_{j,i}\ge 1$, $c_{j,i}^l \in S'$, $c_{j,i}^l=0$ for $l \equiv 0 \pmod{p}$, and $c_{j,i}^{d_{j,i}}\neq 0$. We can define a map from $\ASW\cov^{\eta}_{\underline{\iota},\underline{D}}$ to $L$ by assigning to each covering $\phi_n$ its corresponding set of coefficients $c^l_{j,i}$. Finally, the proposition follows from the fact that $\underline{f}$ and $m \cdot \underline{f}$, where $m$ is multiplicatively invertible in $\mathbb{Z}/p^n$, generate the same cover, as discussed in Proposition \ref{prop1to1ASW}. This implies that $\ASW\cov^{\eta}_{\underline{\iota},\underline{D}}$ is locally isomorphic to the quotient $L/(\mathbb{Z}/p^n)^{\times}$, which establishes the claim.
\end{proof}

\subsection{Irreducible components of the \texorpdfstring{$p$}{p}-rank strata}

\begin{theorem}
The irreducible components of $\ASW\cov_{\underline{\iota},\underline{s}}$ are the strata $\ASW_{\underline{\iota},M}$, with $M \in \Omega_{\underline{\iota}, \underline{s}}$. Suppose $M=(\iota_{j,i})$ is an $r \times n$ matrix. Then the dimension of $\ASW\cov_{\underline{\iota},M}$ is
\begin{equation}
\label{eqndimasw}
    s_n+\sum_{i=1}^n\sum_{j=1}^r \bigg( {\iota_{j,i}}-1 - \bigg\lfloor \frac{\iota_{j,i}-1}{p} \bigg\rfloor \bigg). 
\end{equation}
\end{theorem}

\begin{proof}
The image of $\ASW\cov_{\underline{\iota}, \underline{s}}$ under $B$ is the union of the strata $\mathcal{C}\divisor_{\underline{\iota},M}$ of $\mathcal{C}\divisor^1_{\underline{\iota}}$ with $s(M)=\underline{s}$ by Proposition \ref{propthemapb}. The stratum $\mathcal{C}\divisor_{\underline{\iota},M}$ is irreducible of dimension $s_n$ by Lemma \ref{lemmadimstratamatrix}.

For $M \in \Omega_{\underline{\iota}, \underline{s}}$, consider the morphism $B_M: \ASW\cov_{\underline{\iota}, M} \xrightarrow{} \mathcal{C}\divisor_{\underline{\iota}, M}$. The fiber of $B_M$ over a fixed system of divisors $\underline{D}$ is irreducible by Proposition \ref{propfiberofBM}. By Zariski's main theorem, $\ASW\cov_{\underline{\iota}, M}$ is irreducible since $B_M$ has irreducible fibers and image. Thus, the irreducible components of $\ASW\cov_{\underline{\iota}, \underline{s}}$ are the strata $\ASW\cov_{\underline{\iota}, M}$ with $M \in \Omega_{\underline{\iota}, \underline{s}}$.

The morphism $B_M$ is flat since all its fibers are isomorphic. Thus the dimension of $\ASW\cov_{\underline{\iota}, M}$ is the sum of the dimensions of $\mathcal{C}\divisor_{\underline{\iota},M}$ and of the fibers of $B_M$, which is equal to \eqref{eqndimasw} by Lemma \ref{lemmadimstratamatrix} and Proposition \ref{propfiberofBM}. 
\end{proof}

The below result then follows immediately from the above one and Lemma \ref{lemmacoverandcurve}.

\begin{corollary}
The dimension of $\ASW_{\underline{\iota},M}$ is
\begin{equation}
    s_n+ \sum_{i=1}^n\sum_{j=1}^r \bigg( {\iota_{j,i}}-1 - \bigg\lfloor \frac{\iota_{j,i}-1}{p} \bigg\rfloor \bigg)-3.
\end{equation}
\end{corollary}

\section{Deformation of Artin-Schreier-Witt covers and the geometry of \texorpdfstring{$\ASW_{\underline{\iota}}$}{} }
\label{secdeformation}

\subsection{Deformation of cyclic covers}

One approach to comprehending the geometry of $\ASW_{\underline{\iota}}$ is studying equal-characteristic deformations. In this section, we fixed a $n$-tuple $\underline{\iota}=(\iota_1, \ldots, \iota_n)$ for which $\Omega_{\underline{\iota}}$ is non-empty. For a given $M \in \Omega_{\underline{\iota}}$, we define the associated irreducible stratum as $\Gamma_M$, specifically $\Gamma_M:=\ASW_{\underline{\iota},M}$.

\begin{proposition}
Let $\phi_n$ be a cover that branches at $u$ points $P_1, \ldots, P_u$. Then it can be factored as
\begin{equation*}
\phi_n = \prod_{i=1}^u \phi_{n,i},
\end{equation*}
where each $\phi_{n,i}$ is a cyclic cover of $\mathbb{P}^1_k$ of degree dividing $p^n$ and branches only at $P_i$.
\end{proposition}

\begin{proof}
    For the case $n=1$, the morphism $\phi_1$ can be represented by a rational function $g_1 \in k(x)$. Without loss of generality, we can assume that $g_1$ is in reduced form, which allows us to express it as
\begin{equation*}
g_1 = \sum_{i=1}^u \sum_{j=1}^{m_i} \frac{a_{i,j}}{(x-P_i)^j} = \sum_{i=1}^u f_{i,1},
\end{equation*}
where $a_{i,j} \in k$ and $a_{i,j}=0$ whenever $p \mid j$. It is clear that taking $\phi_{1,i}:=\mathfrak{K}_1(f_{i,1})$ satisfies the desired conditions. 

Now, let's consider the case $n>1$. We proceed by induction on Witt vectors as in the proof of Proposition \ref{propreducedform} to obtain the decomposition
\begin{equation*}
    (g_1, \ldots, g_n) = \sum_{i=1}^u (f_{i,1}, \ldots, f_{i,n}) =: \sum_{i=1}^u \underline{f}_i,
\end{equation*}
where each $\underline{f}_i$ only has $P_i$ as a pole. For further details, refer to \cite[Proposition 5.1]{2020arXiv201013614D}.
\end{proof}

\begin{remark}
The stalk $\phi_{n,P_i}$ of $\phi_n$ at $P_i$ corresponds to a cyclic extension of the ring of power series over $k$. It can be shown that $\phi_{n,P_i}$ is precisely given by $\phi_{n,i}$, that is, $\phi_{n,P_i}=\mathfrak{K}_n(\underline{f}_i)$.
\end{remark}

\begin{proposition}
\label{proplocaltoglobal}
Let $M,N \in \Omega_{\underline{\iota}}$. Suppose $\phi_n \in \Gamma_N$ as described in the previous proposition. Then $\phi_n$ can be deformed to an element in $\Gamma_N$ if and only if each local extension $\phi_{n,P_i}$ can be deformed to an element in $\Gamma_{N_i}$, where $(N_i)_{i=1}^u$ is a partition of $N$. In particular, a $\mathbb{Z}/p^n$ cover with branching datum $M$ can be deformed to $\Phi_n$ with branching datum $N$ only if $M \prec N$.
\end{proposition}

\begin{proof}
``$\Rightarrow$'' We consider $\phi_n$ as a character in $\cohom^1(G_K, \mathbb{Z}/p^n)$, where $K=k(x)$. In order for a $\mathbb{Z}/p^n$-cover $\Phi_n \in \cohom^1(G_{\Frac(R)(x)}, \mathbb{Z}/p^n)$ to be a deformation of $\phi_n$ over a complete discrete valuation ring $R$ in characteristic $p$ with residue field $k$, it is necessary for the sub-$\mathbb{Z}/p^j$-cover $\Phi_i=(\Phi_n)^{p^{n-i}}$ to be a deformation of $\phi_i:=(\phi_n)^{p^{n-i}}$ for $i=1, \ldots, n-1$.

Let $\mathscr{P}=\{ P_1, \ldots, P_u \} \subset \mathbb{P}^1_k$ be the branch locus of $\phi_i$ corresponding to the rows $r_1(M^{p^{n-i}}), \ldots, \allowbreak r_u(M^{p^{n-i}})$ of the matrix $M^{p^{n-i}}$, where $r_j(M)=[M_{j,1},\ldots, M_{j,n}]\in \mathbb{Z}^n$. Now, let $\mathfrak{P}_{j_1}, \ldots, \mathfrak{P}_{j_s} \in \mathbb{P}^1_R$ be the branch points of $\Phi_i$ specializing to $P_j$, which correspond to the rows $j_1, \ldots, j_s$ of $N^{p^{n-i}}$. These points form an $s \times i$ sub-matrix $N_{P_j}^{p^{n-i}}:=[r_{j_1}(N^{p^{n-i}}), \ldots, \allowbreak r_{j_s}(N^{p^{n-i}})]^{\intercal}$ of $N^{p^{n-i}}$. The completion of $\Phi_i$ at $P_j$ induces a local deformation $\Phi_{i, P_j}$ of $\phi_{i,P_j}$, which generically branches at $\mathfrak{P}_{j_1}, \ldots, \mathfrak{P}_{j_s}$ and has branching datum $N^{p^{n-i}}_{P_j}$. According to the vanishing cycle formula \cite{MR904945}, for each $1 \leq i \leq n$, the generic fiber of $\Phi_{i,P_j}$ and $\phi_{i,P_j}$ have the same Swan conductor and hence the same sum of conductors, as indicated by (\ref{eqndifferent}). Consequently, we have the relation
\begin{equation*}
[M_{j,1},\ldots, M_{j,i}] \prec N^{p^{n-i}}_{P_j},
\end{equation*}
which implies that $M \prec N$.

``$\Leftarrow$'' Let's assume that $\Phi_{n,P_i}$ is a deformation of $\phi_{n,P_i}$ over $R$. This deformation can be extended to a deformation $\Phi_{n,i}$ of $\phi_{n,i}$, represented by $\underline{F}_i:=(F_{i,1}, \ldots, F_{i,n}) \in W_n(\Frac(R)(x))$. By following the same reasoning as before, we can deduce that $\Phi_{n,i}$ has the type $[M_{j,1},\ldots, M_{j,n}] \prec N_{P_j}$, where $N_{P_j}$ is a sub-matrix of $N$. Finally, it can be shown that $\Phi_n=\mathfrak{K}_n(\sum_{i=1}^u \underline{F}_{i})$ deforms $\phi_n$.
\end{proof}

Therefore, our focus can be narrowed down to studying the deformations of local extensions, specifically extensions of the ring of power series over $k$.

\begin{definition}
    We say a deformation from a point in $\Gamma_{M}$ to one in $\Gamma_{N}$  has \emph{type} $M \xrightarrow{} N$.
\end{definition}

\begin{example}
\label{examplenonossdeformation}
Let's consider a $\mathbb{Z}/4$-cover $\chi_2$ of $\mathbb{P}^1_{k[[t]]}$ (where $p=2$) defined by
\begin{equation}
\label{order4deformation}
\wp(Y_1,Y_2)=\left(\frac{1}{x^2(x-t^4)}, \frac{1}{x^3(x-t^4)^2(x-t^2)^2}\right).
\end{equation}
Its special fiber $\overline{\chi}_2$ is birationally equivalent to
\[ \wp(y_1,y_2)=\bigg( \frac{1}{x^3}, \frac{1}{x^7} \bigg). \]
Adding $\wp(1/(t^2x), 1/(t^4x^2)+ 1/(t^5(t^2-1)(x-t^2))$ to the right-hand-side of (\ref{order4deformation}) brings it to a Witt vector with the first entry $(t^6-1)/(t^8x)+1/(t^8(x-t^4))$ and the second entry
\begin{equation*}
    \frac{t^8+t^6+t^2+1}{(t^2-1)t^{16}x^2}+ \frac{t^4+1}{t^{16}x}+\frac{1}{t^{16}(t^2-1)^2(x-t^4)^2} + \frac{t^8+t^6+t^4+1}{t^{16}(t^2-1)^3(x-t^4)}+\frac{t^{11}+t^7+t^2+1}{(t^2-1)^3t^{12}(x-t^2)}.
\end{equation*}
Therefore, due to Theorem \ref{theoremcaljumpirred}, generic fiber has upper jumps $(1,2)$ at $0$, $(1,2)$ at $t^4$, and $(0,1)$ at $t^2$. Hence, $\chi_2$ is a deformation of $\overline{\chi}_2$ of type
\[M:= \begin{bmatrix}
4 & 8 \\
\end{bmatrix} \xrightarrow{}  \begin{bmatrix}
   2 & 3 \\
   2 & 3 \\
   0 & 2\\
   \end{bmatrix}=:N. \] 
Note that the $\mathbb{Z}/2$-sub-cover $\chi_1$ is a deformation with a type of $[4] \to [2,2]^{\top}$. 
\end{example}

From this point onwards, our focus will be on studying deformations of type $[M_{1,1}, \ldots, M_{1, n}] \to N$, where $M_{1,1} \neq 0$ and $N$ has more than one row. 

In the following section, we introduce a powerful technique for determining the existence of deformations of a specific type. However, this technique will be directly applied in only a few selected examples throughout this paper.

\subsection{The Hurwitz tree technique} 

Let $\chi \in \cohom^1(G_{\kappa}, \mathbb{Z}/p^n)$ be a one-point cover. The \textit{Hurwitz tree} is a valuable tool for investigating the possibility of deforming $\chi$ with a given type. It is a combinatorial-differential object that takes the shape of the dual graph of a cover's semistable model. Each edge is assigned a \textit{thickness} corresponding to the associated annulus. Additionally, each vertex of the tree encodes the degeneration data of the restriction of $\chi$ to the corresponding closed subdisc, measured by the \textit{refined Swan conductors}. These conductors are described by $2$-tuples of the form $(\delta, \omega)$, where:
\begin{itemize}
    \item $\delta \in \mathbb{Q}_{\ge 0}$ is the \textit{depth Swan conductor},
    \item For $\delta > 0$, $\omega \in \Omega^1_{\kappa}$ is the \textit{differential Swan conductor},
    \item For $\delta = 0$, $\omega \in W_n(\kappa)$ is the \textit{reduction type}.
\end{itemize}

This information not only determines whether the associated cover has good reduction but also identifies the precise locations of the singularities \cite{MR904945}. The concept was initially formulated by Henrio \cite{2000math.....11098H} to address the lifting problem for $\mathbb{Z}/p$ in mixed characteristic. Subsequently, Bouw, Brewis, and Wewers made improvements to the technique \cite{MR2254623}, \cite{MR2534115}. In \cite{2020arXiv201013614D}, the first author extended the concept to the equal-characteristic setting, enabling the use of combinatorics and differential equation tools to study Artin-Schreier-Witt deformations. For the purposes of this manuscript, we will adopt this extended definition.

\begin{example}
\label{examplenonossdeformationHurwitz}
In Example \ref{examplenonossdeformation}, the Hurwitz tree $\mathcal{T}_1$ associated with $\chi_1$ and the Hurwitz tree $\mathcal{T}_2$ associated with $\chi_2$ have the forms presented in Figure \ref{figexamplenonossdeformationtrees}. The thickness of both $e_0$ and $e_1$ is equal to $1$. Additionally, the refined Swan conductors at the vertices is given in Table \ref{tab:degenerationdataexamplenonossdeformation}. The computation of the conductors is straight from the represented Witt vectors thanks to \cite[Proposition 2.8]{MR3726102}.

\tikzstyle{level 1}=[level distance=1.2cm, sibling distance=3.9cm]
\tikzstyle{level 2}=[level distance=1.2cm, sibling distance=2.1cm]
\tikzstyle{level 3}=[level distance=0.8cm, sibling distance=0.9cm]
\tikzstyle{level 4}=[level distance=1.2cm, sibling distance=1cm]
\tikzstyle{bag} = [text width=11.5em, text centered]
\tikzstyle{end} = [circle, minimum width=3pt,fill, inner sep=0pt]
\begin{figure}[ht]
\begin{subfigure}{.5\textwidth}
  \centering
\begin{tikzpicture}[grow=up, sloped]
\node[end, label=left:{$v_0$}]{}
child{
        node[end,label=left:
                    {$v_1$}]{}
    child {
                node[end, label=above:
                    {$0[t^2]$}] {}
                edge from parent
            }
    child {
        node[end,label=left:
                    {$v_2$}]{}        
            child {
                node[end, label=right:
                    {$ 2[t^{4}]$}] {}
                edge from parent
            }
            child {
                node[end, label=left:
                    {$ 2[0]$}] {}
                edge from parent
            }
            edge from parent 
            node[above] {$e_1$}
            node[below]  {$1$}
    }
    edge from parent
    node[above] {$e_0$}
    node[below] {$1$}
    };
\end{tikzpicture}
  \caption{$\mathcal{T}_1$}
  \label{figexamplenonossdeformationtree1}
\end{subfigure}%
\begin{subfigure}{.5\textwidth}
  \centering
\begin{tikzpicture}[grow=up, sloped]
\node[end, label=left:{$v_0$}]{}
child{
        node[end,label=left:
                    {$v_1$}]{}
    child {
                node[end, label=above:
                    {$2[t^2]$}] {}
                edge from parent
            }
    child {
        node[end,label=left:
                    {$v_2$}]{}        
            child {
                node[end, label=right:
                    {$ 3[t^{4}]$}] {}
                edge from parent
            }
            child {
                node[end, label=left:
                    {$ 3[0]$}] {}
                edge from parent
            }
            edge from parent 
            node[above] {$e_1$}
            node[below]  {$1$}
    }
    edge from parent
    node[above] {$e_0$}
    node[below] {$1$}
    };
\end{tikzpicture}
  \caption{$\mathcal{T}_2$}
  \label{figexamplenonossdeformationtree2}
\end{subfigure}%
\caption{Hurwitz trees in Example \ref{examplenonossdeformation}}
\label{figexamplenonossdeformationtrees}
\end{figure}

\begin{table}[ht]
    \centering
\begin{tabular}{ |p{1.8cm}|p{3.7cm}|p{3.7cm}|p{3.7cm}|  }
\hline
Vertices & $\mathcal{T}_1$ & $\mathcal{T}_2$  \\
\hline
\hline
$v_0$ & $\big( 0, \frac{1}{x^{3}} \big)$ & $\big( 0, \big(\frac{1}{x^{3}}, \frac{1}{x^7} \big) \big) $\\
\hline
$v_1$ & $\big( 3, \frac{dx}{x^4} \big)$ & $\big( 7, \frac{dx}{x^{6}(x-1)^{2}} \big)$\\
\hline
$v_2$ & $\big( 6, \frac{dx}{x^{2}(x-1)^{2}} \big)$ & $\big( 12, \frac{dx}{x^{3}(x-1)^{3}} \big)$  \\
\hline
\end{tabular}
    \caption{Degeneration data of the trees in Figure \ref{figexamplenonossdeformationtrees}}
    \label{tab:degenerationdataexamplenonossdeformation}
\end{table}
\end{example}

When $n=1$, the deformation of Artin-Schreier covers in equal-characteristic is completely determined by these trees.

\begin{proposition}[{\cite{2020arXiv200203719D}}]
There exists a deformation of type $[e] \rightarrow [e_1, \ldots, e_s]^{\top}$ if and only if there exists a Hurwitz tree of the same type.
\end{proposition}

When $n>1$, the conditions for the existence are more subtle, as illustrated in the following theorem.

\begin{theorem}
    There exists a deformation of type $M \xrightarrow{} N$, where $M \in \mathbb{M}_{1 \times n}(\mathbb{Z}_{\ge 0})$ only if there exists $n$ Hurwitz trees $\mathcal{T}_1, \ldots, \mathcal{T}_n$ that satisfy the following conditions:
    \begin{enumerate}
        \item $\mathcal{T}_i$ has type $M^{p^{n-i}} \xrightarrow{} N^{p^{n-i}}$ and
        \item $\mathcal{T}_i$ extends $\mathcal{T}_{i-1}$ for $i=2, \ldots, n$.
    \end{enumerate}
\end{theorem}

The concepts of type and tree extension are introduced in \cite[Definition 4.2 and Definition 4.6]{2020arXiv201013614D}, respectively. Notably, a tree $\mathcal{T}_i$ is an extension of $\mathcal{T}_{i-1}$ only if the Swan conductors of the shared vertices satisfy the conditions stated in \cite[Theorem 3.4.2]{2020arXiv201013614D}.

\begin{example}
In Example \ref{examplenonossdeformationHurwitz}, it can be verified that $\mathcal{T}_2$ is an extension of $\mathcal{T}_1$. For instance, at the vertex $v_2$, the depth Swan conductor in $\mathcal{T}_2$ is precisely twice that in $\mathcal{T}_1$, which implies that the Cartier operator of the differential Swan conductor in $\mathcal{T}_2$ is equal to that in $\mathcal{T}_1$.
\end{example}

\begin{example}
    Let's consider the case where $p=5$. Assume $\chi_2$ represents a deformation of a cover of $\mathbb{P}^1_k$ with the given type:
    \[\begin{bmatrix}
    7 & 35 \\
    \end{bmatrix} \xrightarrow{}  \begin{bmatrix}
   3 & 17 \\
   4 & 18 \\
   \end{bmatrix}. \] 
   Set $\chi_1:=\chi_2^5$. Then the associated trees $\mathcal{T}_1$ and $\mathcal{T}_2$ must have the following forms, where $\epsilon$ represents the thickness of $e_0$.
       \tikzstyle{level 1}=[level distance=1.2cm, sibling distance=3.9cm]
\tikzstyle{level 2}=[level distance=0.8cm, sibling distance=1.5cm]
\tikzstyle{bag} = [text width=11.5em, text centered]
\tikzstyle{end} = [circle, minimum width=3pt,fill, inner sep=0pt]
\begin{figure}[ht]
\begin{subfigure}{.5\textwidth}
  \centering
\begin{tikzpicture}[grow=up, sloped]
\node[end, label=left:{$v_0$}]{}
child{
        node[end, label=left:{$v_1$}]{}
    child {
                node[end, label=right:
                    {$3[P_1]$}] {}
                edge from parent
            }
    child {
                node[end, label=left:
                    {$4[P_2]$}] {}
                edge from parent
            }
    edge from parent
    node[above] {$e_0$}
    node[below] {$\epsilon$}
    };
\end{tikzpicture}
  \caption{$\mathcal{T}_1$}
\end{subfigure}%
\begin{subfigure}{.5\textwidth}
  \centering
\begin{tikzpicture}[grow=up, sloped]
\node[end, label=left:{$v_0$}]{}
child{
        node[end, label=left:{$v_1$}]{}
    child {
                node[end, label=right:
                    {$17[P_1]$}] {}
                edge from parent
            }
    child {
                node[end, label=left:
                    {$18[P_2]$}] {}
                edge from parent
            }
    edge from parent
    node[above] {$e_0$}
    node[below] {$\epsilon$}
    };
\end{tikzpicture}
  \caption{$\mathcal{T}_2$}
\end{subfigure}%
\end{figure}

The criteria for Hurwitz trees assert that the depth at $v_1$ in $\mathcal{T}_1$ is $\delta_1(v_1) = 6\epsilon$, while the depth at $v_1$ in $\mathcal{T}_2$ is $\delta_2(v_1) = 34\epsilon$. Observe that $\delta_2(v_1)$ exceeds $5\delta_1(v_1)$. According to \cite[Theorem 3.4.2 (ii) (b)]{2020arXiv201013614D}, for $\mathcal{T}_1$ to be an extension of $\mathcal{T}_2$, the differential form at $v_1$ in $\mathcal{T}_2$ must be exact.

Hence, this implies the existence of an exact differential form $\omega = \frac{dx}{x^{17}(x-a)^{18}}$, where $a \in k^{\times}$, represents an element in the non-zero set of the field $k$. However, applying the Gr{\"o}bner Basis technique \cite[\S 6.1]{2020arXiv200203719D} leads to the conclusion that $\omega$ can never be exact. Consequently, achieving a deformation of the desired type is not possible.
\end{example}

\subsection{Essential jumps and a natural family of deformations}
\label{secessentialdeform}
Similar to the scenario with Artin-Schreier curves \cite[\S 4.1]{MR2985514}, a collection of non-trivial yet ``natural'' deformations exists in this context.

Suppose $\overline{\chi} \in \cohom^1(G_\mathbb{K}, \mathbb{Z}/p^n)$ gives a $\mathbb{Z}/p^n$-cover that branches at a single point with higher ramification jumps $\iota_1<\iota_2<\ldots<\iota_n$. Let $\iota_0=0$. For each $1 \le j \le n$, we can express $\iota_j -p\iota_{j-1}=pq_j+\epsilon_j$, where $0 \le q_j$ and $0 \le \epsilon_j <p$. It follows that $0< \epsilon_j$ if and only if $(p, \iota_j)=1$ if and only if $p \iota_{j-1}<\iota_j$. The term $q_j$ is referred to as the \textit{essential part} of the upper jump at position $j$, and if $0<q_j$, we say that $\iota_j$ is an \textit{essential upper jump}. Let $q:=1+\sum_{j \text{ essential}} q_j$. According to Pop, it is possible to partition $(\iota_1+1, \ldots, \iota_{n}+1)$ into an $q \times n$ matrix $M:=(e_{j,i})$ with no essential parts as follows:
\begin{enumerate}
    \item $e_{1,i}=pe_{1,i-1}+\epsilon_i$ for $1 \le i \le n$, $e_{1,0}=0$,
    \item Add $q_j$ number of rows as follows to $M$ for each essential place $j$
    \begin{equation}
    \label{eqnrowqj}
        [0, \ldots, 0, p-1, p^2-1, \ldots, p^{n-j+1}-1].
    \end{equation}
\end{enumerate}
Pop demonstrates the existence of a deformation of type $M$, which we call a deformation of \textit{Pop type}. We restate his result here, adapting it to the conventions used in this paper.

\begin{lemma}[{\cite[Key Lemma 3.2]{MR3194816}}]
\label{lemmapopdeformation}
Let $A=k[[x]] \xhookrightarrow{} k[[z]]:=B $ be a cyclic $\mathbb{Z}/p^n$-extension with branching datum $(\iota_1, \ldots, \iota_n)$ and $\delta_0:=\lfloor \iota_n/(p-1) \rfloor$. In the above notation, let $P_1, \ldots, P_r \in \mathfrak{m}$ be distinct elements that are $p^{\delta_0}$-powers. Then there exists a $\mathbb{Z}/p^n$-deformation of $A \xhookrightarrow{} B$ over $k[[\varpi]]$ of types $M$ that has $\{P_1, \ldots, P_r\}$ as branch locus.
\end{lemma}

\begin{example}
\label{examplepopdeformation}
Consider $\overline{\chi}$, a $\mathbb{Z}/5^2$-cover of $\mathbb{P}^1_k$ defined by
\[ \wp(y_1,y_2)=\bigg( \frac{1}{x^8}, \frac{1}{x^{52}}+\frac{1}{x^{46}} \bigg)=\bigg( \frac{1}{x^8}, \frac{1+x^6}{x^{52}}\bigg) \]
This cover branches at a single point with jumps $(\iota_1, \iota_2)=(8, 52)$, giving a conductors $[9,53]$. We have the following decompositions:
\begin{equation}
\begin{split}
\iota_1-5\iota_0 & = 5 \cdot 1 + 3 \\
\iota_2-5\iota_1 & = 5 \cdot 2+2.
\end{split}
\end{equation}
Based on this data, we can ``split'' $[9,53]$ into the $4 \times 2$ matrix:
\begin{equation}
\label{eqnpopsplit}
\begin{bmatrix}
9 & 53 \
\end{bmatrix} \xrightarrow{} \begin{bmatrix}
3+1 & 3\cdot 5+2+1 \\
5 & 25 \\
0 & 5\\
0 & 5\\
\end{bmatrix}=\begin{bmatrix}
4 & 18 \\
5 & 25 \\
0 & 5\\
0 & 5 \\
\end{bmatrix}.
\end{equation}
We can show that the cover $\chi$ defined by
\[ \wp(Y_1,Y_2)= \bigg(\frac{1}{x^3(x-t_1)^5}, \frac{1+x^6}{x^{17}(x-t_1)^{25}(x-t_2)^{5}(x-t^3)^{5}}  \bigg) \]
is a deformation of $\overline{\chi}$ over $k[[t_1,t_2,t_3]]$ of type (\ref{eqnpopsplit}).
\end{example}

\begin{remark}
When the fibers of a Pop type deformation are Artin-Schreier covers, the deformation is also referred to as an \textit{Oort-Sekiguchi-Suwa deformation} (OSS deformation). These deformations extend to mixed characteristic and give rise to the \textit{Oort-Sekiguchi-Suwa component} (OSS component) \cite[\S 4.3]{MR1767273}, which is a sub-scheme of co-dimension $0$ within the versal deformation space of an Artin-Schreier cover.
\end{remark}

We then obtain an analog of \cite[Proposition 4.2]{MR2985514}.

\begin{proposition}
Let $Y_n$ be an $n$-Artin-Schreier-Witt $k$-curve with branching datum $N$ \eqref{eqnmatrixbranchingdatum} and $p$-rank $Q$. Suppose there is a row, without loss of generality the first one, $r_1(N) = [\iota_1, \ldots, \iota_n]$ with essential parts as described in the setup. Then $Y_n$ can be deformed to a curve with a strictly higher $p$-rank, where the difference in $p$-rank is determined by the essential parts in $r_1(N)$.
\end{proposition}

\begin{proof}
    Lemma \ref{lemmapopdeformation} gives a deformation $\Phi: \mathscr{Y}_n \xrightarrow{} \mathbb{P}^1_{k[[t]]}$ of type $$[r_1(N), \ldots, r_s(N)]^{\top} \rightarrow [r_{1}(M), \ldots, r_{q}(M), r_2(N), \ldots, r_s(N)]^{\top}.$$ 
    Furthermore, according to \eqref{eqnprank}, each row of type \eqref{eqnrowqj} contributes an additional $\sum_{i=j}^n p^{n-i} (p^i-1)$ to the $p$-rank of $\mathscr{Y}_{n, \eta}$. Therefore, the generic fiber has a $p$-rank given by
    \begin{equation*}
        Q + \sum_{j \text{ essential}} \sum_{i=j}^n p^{n-i} (p^i-1) q_j.
    \end{equation*}
    This completes the proof.
\end{proof}

We can now study the irreducibility of the moduli space. Recall that $G_{\underline{\iota}}$, introduced at the end of \S \ref{secpartitionmatrix}, denotes the graph of partitions of $\Omega_{\underline{\iota}}$. 

\begin{proposition}
\label{propclosuredeformation}
Suppose $M, N \in \Omega_{\underline{\iota}}$. Then $\Gamma_{M}$ lies in the closure of $\Gamma_{N}$ in $\ASW_{\underline{\iota}}$ if and only if any point in $\Gamma_{M}$ can be deformed over $k[[t]]$ to a point in $\Gamma_{N}$.
\end{proposition}

\begin{proof}
The proof can be found in \cite[Proposition 3.4]{DANG2020398}.
\end{proof}

\begin{theorem}
\label{theoremirredcomponent}
The irreducible components of $\ASW_{\underline{\iota}}$ are in a one-to-one correspondence with the vertices of the graph $G_{\underline{\iota}}$ with no essential parts. Specifically, if $M \in \Omega_{\underline{\iota}}$ has no essential parts, then the irreducible component associated with $M$ is the closure of $\ASW_{\underline{\iota},M}$.
\end{theorem}

\begin{proof}
Lemma \ref{lemmapopdeformation} and Proposition \ref{propclosuredeformation} imply that any stratum $\Gamma_M$ where $M$ has essential parts is contained within the closure of at least one stratum that does not have essential parts. Furthermore, Proposition \ref{proplocaltoglobal} indicates that the points in $\Gamma_N$, where $N$ lacks an essential part, cannot be further deformed non-trivially. Therefore, the theorem immediately follows from these observations.
\end{proof}

\begin{corollary}
If $G_{\underline{\iota}}$ is disconnected, then $\ASW_{\underline{\iota},M}$ is also disconnected.
\end{corollary}

\begin{proof}
Suppose ${M_1, \ldots, M_r}$ and ${N_1, \ldots, N_s}$ are two sets of minimum vertices of $G_{\underline{\iota}}$ that lie in different (union of) connected components in the graph. According to Proposition \ref{proplocaltoglobal}, the union of the closures of the $M_j$'s does not intersect any of the $N_i$'s. 
\end{proof}

\subsection{Irreducibility of the Artin-Schreier-Witt locus}

\begin{proposition}
The moduli space $\ASW_{(\iota_1, \ldots, \iota_n)}$ is irreducible if and ony if the following conditions are met:
\begin{enumerate}[label*=\arabic*.]
    \item $\iota_1=2$ or $\iota_1=3$, and
    \item $\iota_{i}=p\iota_{i-1}-p+1$ or $\iota_{i}=p\iota_{i-1}-p+2$ for $2 \le i \le n$.
\end{enumerate}
\end{proposition}

\begin{proof}
By applying Theorem \ref{theoremirredcomponent}, we establish a one-to-one correspondence between the irreducible components of $\ASW_{({\iota_1, \ldots, \iota_n})}$ and the minimum vertices of $G_{({\iota_1, \ldots, \iota_n})}$. Therefore, it is sufficient to determine the $G_{({\iota_1, \ldots, \iota_n})}$ graphs with only one minimum vertex, which are precisely those satisfy the given conditions.
\end{proof}

\section{The connectedness of \texorpdfstring{$\ASW_{\overline{\iota}}$}{}}
\label{secconnectedness}

For each $m\ge n$, there is canonical map
\begin{equation*}
\begin{split}
     \mathscr{X}:= \ASW_{(\iota_1, \ldots, \iota_{n-1}, \iota_n, \ldots, \iota_m)} & \xrightarrow{\Theta_{m/n}} \ASW_{(\iota_1, \ldots, 
     \iota_{n-1}, \iota_n)}:=\mathscr{Y}\\
    \phi_m & \longmapsto{} (\phi_m)^{p^{m-n}}  \\
\end{split}
\end{equation*}
which, on the level of points, associates each $\mathbb{Z}/p^m$-cover represented by a length-$m$ Witt vector $(f_1, \ldots, f_{n-1},\allowbreak f_n, \ldots, f_m)$ with the $\mathbb{Z}/p^n$-sub-cover corresponding to $(f_1, \ldots, f_{n-1},\allowbreak f_n)$.

\begin{proposition}
With the above notations, $\Theta_{m/n}$ is a surjective, continuous, and closed morphism of $k$-schemes.
\end{proposition}

\begin{proof}
The proof that $\Theta_{m/n}: \mathscr{X} \xrightarrow{} \mathscr{Y}$ is morphism, hence continuous, is straight-forward. Suppose $\phi_m \in \ASW_{(\iota_1, \ldots, \iota_{n-1}, \iota_n, \ldots, \iota_m)}$ and $\phi_n=(\phi_m)^{p^{m-n}}$. According to \cite[Theorem 1.2]{2020arXiv201013614D}, any deformation $\Phi_n$ of $\phi_n$ over $k[[t]]$ extends to a deformation of $\phi_m$. In other words, for each point $x \in \mathscr{X}$, if $y=f(x)$ lies in the closure of a point $v \in \mathscr{Y}$, then $x \in \overline{f^{-1}(v)}$. Applying the theorem to the case where $\Phi_n$ is trivial establishes the surjectivity.

Now suppose $A\subset \mathscr{X}$ is closed, and $y \in
(f(A))^c$ is a limit point of $f(A)$. If $\dim (A)>0$, then according to \cite[Proposition 3.2]{DANG2020398}, we can assert that $y \in \overline{v}$ for a point $v \in f(A)$. If $\dim A=0$, then $A$ is a finite union of closed points, and we can conclude the same thing. Moreover, if $v=f(u)$ for a point $u \in A$, then there exists a model of $u$ over $k[[t]]$ whose special fiber is isomorphic to a smooth point $x \in f^{-1}(y)$. Hence, $x$ is a limit point of $u \in A$. However, $x \not\in A$ since $y \not\in f(A)$. This contradicts the assumption that $A$ is closed. Therefore, $f(A)$ is a closed subset of $\mathscr{Y}$.
\end{proof}

\begin{lemma}
Suppose $m \ge n$ are integers, and $\ASW_{(\iota_1, \ldots, \iota_{n-1}, \iota_n)}$ is disconnected. Then so is $\ASW_{(\iota_1, \ldots, \iota_{n-1}, \iota_n, \ldots, \iota_m)}$.
\end{lemma}

\begin{proof}
The lemma follows directly from the continuity of $\Theta_{m/n}$, which guarantees that connected sets are mapped to connected sets.
\end{proof}

By utilizing \cite{2020arXiv200203719D} and \cite{DANG2020398}, we can derive some results related to disconnectedness.

\begin{corollary}
    For $p \geq 5$ and $3 \leq \iota_1 \leq 2p-2$, the moduli space $\ASW_{(\iota_1, \ldots, \iota_n)}$ is disconnected.
\end{corollary}

\begin{proof}
We recall that an Artin-Schreier cover of conductor $\iota_1$ has genus $g_1 = (\iota_1-2)(p-1)/2$. By \cite[Theorem 1.3]{2020arXiv200203719D}, the moduli space $\ASW_{(\iota_1)} = \mathcal{AS}_{g_1}$ is disconnected when $\frac{p-1}{2} < g_1 \leq (p-1)(p-2)$. Since $3 \leq \iota_1 \leq 2p-2$, we have $\frac{p-1}{2} < g_1 \leq (p-1)(p-2)$. Thus, the corollary follows.
\end{proof}

\subsection{The connectedness}
Construct a directed graph $C_{\underline{\iota}}$ as follows: The vertices of the graph correspond to the elements of $\Omega_{\underline{\iota}}$. In $C_{\underline{\iota}}$, there exists an edge from vertex $M$ to vertex $N$ if and only if there is a deformation from a curve in $\Gamma_M$ to a curve in $\Gamma_N$. Consequently, the graph $C_{\underline{\iota}}$ is connected if and only if the moduli space $\ASW_{\underline{\iota}}$ is also connected.

\begin{remark}
When $\underline{\iota}$ has length one, we draw an edge from vertex $M$ to vertex $N$ when the closure of $\Gamma_N$ contains $\Gamma_M$ \cite[\S 3.2]{DANG2020398}. That is because the closure condition is equivalent to the existence of a deformation from a point in $\Gamma_M$ to one in $\Gamma_N$ \cite[Corollary 3.6]{DANG2020398}. However, it is currently unknown whether this result holds in general.
\end{remark}

\begin{example}
Let's examine the connectedness of the moduli space $\ASW_{(4,8)}$ in characteristic $2$.

\begin{proposition}
\label{propnontrivialdef}
    Suppose $p=2$, $M=[4,8]$, and 
    \begin{equation*}
        N=\begin{bmatrix}
        2 & 2 & 0 \\
        3 & 3 & 2 \\
        \end{bmatrix}^{\top}.
    \end{equation*}
    Then any cover of type $M$ deforms to one with type $N$.
\end{proposition}

\begin{proof}
Let $\phi \in \Gamma_M$ be given. Without loss of generality, we can assume that $\phi$ is described by the affine equation
\begin{equation*}
    \wp(y_1, y_2)=\bigg( \frac{1}{x^3}, \frac{a_0+a_2x^2+a_4x^4+a_6x^6}{x^7} \bigg) \in W_n(k(x)),
\end{equation*}
where $a_i \in k$ and $a_0 \neq 0$. where $a_i \in k$ and $a_0 \neq 0$. By employing a similar approach as in Example \ref{examplenonossdeformation}, we find that the following Witt vector representing $\Phi \in \cohom^1(G_{k(t)(x)}, \mathbb{Z}/4)$ serves our purpose:
\begin{equation*}
    \bigg( \frac{1}{x^2(x-t^4)}, \frac{a_0+a_2x^2+a_4x^4+a_6x^6}{x^3(x-t^4)^2(x-\sqrt{a_0}t^2)^2} \bigg),
\end{equation*}
where $\sqrt{a_0}$ denotes the solution to $Z^2-a_0$ in $k$.
\end{proof}

\begin{remark}
    One can gain geometric insight into the above construction by employing the Hurwitz tree technique discussed in \cite{2020arXiv201013614D}. 
\end{remark}

\begin{corollary}
    The moduli space $\ASW_{(4,8)}$ is connected. 
\end{corollary}

\begin{proof}
The set $\Omega_{(4,8)}$ consists of the matrices $M$ and $N$ described in Proposition \ref{propnontrivialdef}, as well as the matrix
\begin{equation*}
Q=\begin{bmatrix}
2 & 4 \\
2 & 4 \\
\end{bmatrix}.
\end{equation*}
According to the same proposition, $\Gamma_M$ lies in the closure of $\Gamma_N$, which is itself an irreducible component due to Theorem \ref{theoremirredcomponent}. Additionally, Lemma \ref{lemmapopdeformation} states that any cover of type $M$ can be deformed into a cover of type $Q$. Therefore, $\Gamma_M$ is the intersection of the closures of $\Gamma_N$ and $\Gamma_Q$, which establishes the corollary.
\end{proof}

\end{example}

In \cite{DANG2020398}, we demonstrate that $\ASW_{\iota_1}$ is connected (for any characteristic $p$) when $\iota_1$ is sufficiently large. Based on our computed examples, we expect that the same holds true for taller towers of cyclic coverings.

\begin{question}
    Is the moduli space $\ASW_{\underline{\iota}}$ is connected when $\underline{\iota}$ is sufficiently large?
\end{question}

\bibliographystyle{alpha}
\bibliography{bib}

\end{document}